\documentclass[12pt]{amsart}
\usepackage{amsmath,amssymb,latexsym,cancel,rotating}

\numberwithin{equation}{section}

\newtheorem{theorem}{Theorem}[section]
\newtheorem{proposition}[theorem]{Proposition}

\newtheorem{corollary}[theorem]{Corollary}
\newtheorem{lemma}[theorem]{Lemma}

\theoremstyle{definition}

\newtheorem{remark}[theorem]{Remark}
\newtheorem{example}[theorem]{Example}
\newtheorem{definition}[theorem]{Definition}

\input xy
\xyoption{poly}
\xyoption{2cell}
\xyoption{all}

\def\ZZ{\mathbb{Z}}

\def\QQ{\mathbb{Q}}

\def\Fcal{\mathcal{F}}

\renewcommand{\eqref}[1]{{\rm (\ref{#1})}}

\begin{document}

\title[A  quantum analogue of generalized cluster algebras]
{A quantum analogue of generalized cluster algebras}

\author{Liqian Bai, Xueqing Chen, Ming Ding and Fan Xu}
\address{Department of Applied Mathematics, Northwestern Polytechnical University, Xi'an, Shaanxi 710072, P.R. China}
\email{bailiqian@nwpu.edu.cn (L.Bai)}
\address{Department of Mathematics,
 University of Wisconsin-Whitewater\\
800 W.Main Street, Whitewater,WI.53190.USA}
\email{chenx@uww.edu (X.Chen)}
\address{School of Mathematical Sciences and LPMC,
Nankai University, Tianjin, P.R.China}
\email{m-ding04@mails.tsinghua.edu.cn (M.Ding)}
\address{Department of Mathematical Sciences\\
Tsinghua University\\
Beijing 100084, P.~R.~China} \email{fanxu@mail.tsinghua.edu.cn
(F.Xu)}


\thanks{Ming Ding was supported by NSF of China (No. 11301282) and Specialized Research Fund for the Doctoral Program of Higher Education (No. 20130031120004) and Fan Xu was supported by NSF of China (No. 11071133).}




\maketitle

\begin{abstract}
We define a quantum analogue of a class of generalized cluster algebras which can be viewed as a generalization of
quantum cluster algebras defined in \cite{berzel}. In the case of rank two, we  extend some structural results from the classical theory of generalized cluster algebras obtained in \cite{CS}\cite{rupel} to the quantum case.
\end{abstract}


\section{Background}

Fomin and Zelevinsky invented cluster algebras \cite{ca1}\cite{ca2} in order
to study total positivity in algebraic groups and canonical bases in quantum groups. As a natural generalization,
Chekhov and Shapiro introduced the generalized cluster algebras in \cite{CS}.
The main difference between cluster algebras and generalized cluster algebras is that the  binomial exchange relation for cluster
variables of  cluster algebras is replaced by the polynomial with arbitrary positive degree for those cluster
variables of generalized cluster algebras. In \cite{CS}, it was shown that the generalized cluster algebras have the Laurent phenomenon, which is regarded as the most important feature of the classical cluster algebras. Some other important properties of cluster algebras also hold true in the generalized cluster algebras \cite{CS}\cite{nak}\cite{rupel}.

In order to study the canonical basis in an algebraic framework, Berenstein and Zelevinsky \cite{berzel} introduced  quantum cluster algebras as a noncommutative analogue of cluster algebras in which the cluster variables are  quantized. The other quantum deformation of  cluster algebras given  by Fock and Goncharov \cite{fg1}\cite{fg2} is realized through
quantizing the coefficients.
Recently, Nakanishi extends the notion of the quantization of the coefficients of
the classical cluster algebras to the generalized cluster algebras, and shows that it is tightly integrated with
the quantum dilogarithms of higher degrees \cite{nak0}. Motivated by these, a natural problem is that what is the notion of
the quantization of the cluster variables in generalized cluster algebras, and do the essentially important properties of (quantum) cluster algebras still hold true in these cases?

The aim of this note is to propose the notion of generalized quantum cluster algebras, which can be considered as a quantum analogue of  a class of generalized cluster algebras  introduced in \cite{CS} and  can be also viewed as a generalization of quantum cluster algebras introduced in \cite{berzel}. In the case of rank two, we  prove some structural results  of these generalized quantum cluster algebras such as the Laurent phenomenon. The method is to use the tools developed in  \cite{berzel}\cite{rupel}. However, we do not know whether the
positivity property does hold or not. Note that in ~ \cite{usn}, Usnich proved the Laurent phenomenon in  the noncommutative analogue for  only two variables with monic and palindromic polynomials $P_1(x)=P_2(x)$.

\section{Generalized cluster algebras of geometric type}

Firstly, we recall the definition of generalized cluster algebras (see \cite{CS} and \cite{gsv16}). Let $m$ and $n$ be two positive integers with $m\geq n$. Suppose that $\widetilde{B}=(b_{ij})_{m\times n}=\begin{bmatrix}B\\C
\end{bmatrix}$ be an $m\times n$ integer matrix whose principal
part $B$ is a $n\times n$ skew-symmetrizable matrix. For each $k\in\{1,\ldots,n\}$, there is a positive integer $d_k$ such that $d_k|b_{jk}$ for all $1\leq j\leq n$. We set
\begin{equation}\label{def of d}
\textbf{d}=(d_1,\ldots,d_n)
\end{equation}
and
\begin{equation}\label{def of beta}
\beta_{ij}=\begin{cases}
                   \displaystyle\frac{b_{ij}}{d_j} & \text{ if $1\leq j\leq n$, }\\   \\
                   \displaystyle\lfloor\frac{b_{ij}}{d_j}\rfloor & \text{ if $n+1\leq j\leq m$.}
                  \end{cases}
\end{equation}
Here $\lfloor x\rfloor$ denotes the largest integer $l$ that is less than or equal to $x$.

We denote $[b_{ij}]_+=b_{ij}$ if $b_{ij}\geq0$, and $[b_{ij}]_+=0$ if $b_{ij}\leq 0$.
Let $\Fcal:=\QQ(x_1,\ldots,x_m)$ be a field of rational functions in $m$ independent variables. The generalized cluster algebra which will be introduced is a subring of $\Fcal$.

Let $\mathbb{P}=\mathbb{Z}[x_{n+1},\cdots,x_{m}]$. For each $1\leq i\leq n$, the $i$-th string $\rho_i$ is a collection of monomials $\rho_{i,r}\in\mathbb{P}$, $0\leq r\leq d_i$ satisfying that $\rho_{i,0}=\rho_{i,d_i}=1.$
\begin{definition}\label{generalized seed}
A generalized seed in $\Fcal$  is a triple $(\widetilde{\textbf{x}},\rho,\widetilde{B})$, where
\begin{enumerate}
  \item[(1)] $\widetilde{\textbf{x}}=\{x_1,\cdots,x_{m}\}$ is a transcendence basis of $\Fcal$, called a extended cluster, and the tuple $\textbf{x}=\{x_1,\cdots,x_{n}\}$ called a cluster whose elements are called cluster variables;
  \item[(2)] the matrices $B$  and $\widetilde{B}$ defined above are called the exchange matrix and the extended exchange
matrix, respectively;
  \item[(3)] a coefficient tuple is a  $n$-tuple of strings $\rho=(\rho_1,\cdots,\rho_n)$, where the $i$-th string $\rho_i$ is defined above for each $1\leq i\leq n$.
  \end{enumerate}
\end{definition}

As an easy consequence of the definition of the matrix mutation, we obtain that  $d_k|b'_{jk}$ for all $1\leq j\leq n$. We now turn to the definition of the seed mutation.

\begin{definition}\label{mutation}
(Seed mutation) Let $(\widetilde{\textbf{x}},\rho,\widetilde{B})$ be a generalized seed in $\Fcal$. For each $1\leq k\leq n$, the mutation of $(\widetilde{\textbf{x}},\rho,\widetilde{B})$ in direction $k$ is another generalized seed $\mu_k(\widetilde{\textbf{x}},\rho,\widetilde{B}):=(\widetilde{\textbf{x}}^{\prime},\rho^{\prime},\widetilde{B}^{\prime})$, where
\begin{enumerate}
  \item[(1)] the entries of the matrix $\widetilde{B}^{\prime}=(b^{\prime}_{ij})$ are given by
  \begin{equation*}
   b^{\prime}_{ij}=\begin{cases}
                   -b_{ij} & \text{ if $i=k$ or $j=k$,}\\   \\
                   b_{ij}+\frac{1}{2}(|b_{ik}|b_{jk}+b_{ik}|b_{kj}|) & \text{otherwise;}
                  \end{cases}
  \end{equation*}
  \item[(2)] $\widetilde{\textbf{x}}^{\prime}=(\widetilde{\textbf{x}}\setminus\{x_k\}) \cup x_{k}^{\prime}$, where the new variables $x_{k}^{\prime}\in\Fcal$ is given as follows
  \begin{equation*}
   x^{\prime}_{i}=\begin{cases}
                   x_{i} & \text{ if $i\neq k$,}\\   \\
                   x_{k}^{-1}(\sum\limits_{s=0}^{d_k}\rho_{k,s}\prod\limits_{j=1}^{m} x_{j}^{\beta_{jk}s+[-b_{jk}]_{+}}) & \text{otherwise;}
                  \end{cases}
  \end{equation*}
  \item[(3)] $\rho^{\prime}=(\rho^{\prime}_1,\cdots,\rho^{\prime}_n)$, where in each $i$-th string $\rho^{\prime}_i$, $1\leq i\leq n$,  the components are given by
   \begin{equation*}
   \rho^{\prime}_{i,s}=\begin{cases}
                   \rho_{i,d_i-s} & \text{ if $i=k$,}\\   \\
                   \rho_{i,s} & \text{ otherwise. }
                  \end{cases}
  \end{equation*}
\end{enumerate}
\end{definition}

Note that for each $k$, $1\leq k\leq n$, the seed mutation $\mu_k$  is an involution, i.e,
$\mu_k(\mu_k(\widetilde{\textbf{x}},\rho,\widetilde{B}))=(\widetilde{\textbf{x}},\rho,\widetilde{B}).$

The seed $(\widetilde{\textbf{x}}^{\prime},\rho^{\prime},\widetilde{B}^{\prime})$ is said to be mutation-equivalent to the seed $(\widetilde{\textbf{x}},\rho,\widetilde{B})$, if $(\widetilde{\textbf{x}}^{\prime},\rho^{\prime},\widetilde{B}^{\prime})$ can be obtained from
$(\widetilde{\textbf{x}},\rho,\widetilde{B})$ by a sequence of seed mutations, i.e,
$$
(\widetilde{\textbf{x}}^{\prime},\rho^{\prime},\widetilde{B}^{\prime})=\mu_{k_t}(\ldots(\mu_{k_1}(\widetilde{\textbf{x}},\rho,\widetilde{B}))\ldots)
$$
for some $1\leq k_1,\ldots,k_t\leq n$. This operation gives an equivalence relation on seeds. Then one can define the generalized cluster algebras as follows.

\begin{definition}
Let $\mathcal{S}$ be a set consisting of all generalized seeds in $\Fcal$ which are mutation-equivalent to the initial seed $(\widetilde{\textbf{x}},\rho,\widetilde{B})$. The generalized cluster algebra $\mathcal{A}(\widetilde{\textbf{x}},\rho,\widetilde{B})$ is the $\ZZ[x_{n+1}^{\pm1},\ldots,x_{m}^{\pm1}]$-subalgebra of $\Fcal$ generated by all cluster variables from all seeds in $\mathcal{S}$.
\end{definition}

As in the classical theory of cluster algebras, the generalized cluster algebras also possess the Laurent phenomenon.

\begin{theorem}
(\cite{CS}) Each  cluster variable in $\mathcal{A}(\widetilde{\textbf{x}},\rho,\widetilde{B})$ is a Laurent polynomial in the initial
cluster variables.
\end{theorem}

\section{A quantum analogue of generalized cluster algebras}

In this section, we will give a quantum deformation of a class of generalized cluster algebras in which
\begin{enumerate}
  \item[(1)] the coefficients $\rho_{i,r}$ are integers and satisfy
$$\rho_{i,r}=\rho_{i,d_i-r}$$
for each $i$ and $r$, where $1\leq i\leq n$ and $0\leq r\leq d_i$.
  \item[(2)] for each $1\leq k\leq n$, the above positive integers $d_k$ satisfy $d_k|b_{jk}$ for all $1\leq j\leq m$.
\end{enumerate}
\begin{remark}
It is easy to see that the coefficients in this class of generalized cluster algebras do not change under mutation.
\end{remark}

\begin{definition}\label{Lambda B}
Let $\widetilde{B}=(b_{ij})_{}$ be an $m\times n$ integer matrix with $m\geq n$ and $\Lambda=(\lambda_{ij})$ be an $m\times m$ skew-symmetric integer matrix. The pair $(\Lambda,\widetilde{B})$ is said to be compatible if we have
$$-\Lambda\widetilde{B}=\begin{bmatrix}D\\0
\end{bmatrix}$$ where $D=\text{diag}\{\widetilde{d}_1,\ldots, \widetilde{d}_n\}$ is an $n\times n$ diagonal matrix with positive integers diagonal entries $\widetilde{d}_i$, $1\leq i\leq n.$
\end{definition}

By \cite[Proposition 3.3]{berzel}, if the pair $(\Lambda,\widetilde{B})$ is compatible, then  the matrix $\widetilde{B}=\begin{bmatrix}B\\C
\end{bmatrix}$ has full rank and the product matrix
$DB$ is skew-symmetric.

Let the sign $\varepsilon$ be an element in $\{+1, -1\}$. For each $k\in\{1,2,\ldots,n\}$ and each sign $\varepsilon$, by  the discussion in \cite{bfz}, the matrix $\widetilde{B}^{\prime}=\mu_k(\widetilde{B})$ can be rewritten as follows
$$
\widetilde{B}^{\prime}=E_{\varepsilon}\widetilde{B}F_{\varepsilon},
$$
where
\begin{enumerate}
  \item[(1)] $E_{\varepsilon}=(e_1,e_2,\ldots,e_{k-1},e^{\prime}_{k},e_{k+1},\ldots,e_m)$ is the square matrix of degree $m$ with
  $$e_i=(0,\ldots,0,1,0,\ldots,0)^T\in\ZZ^{m}$$
      for $i\neq k$, and
      $$e^{\prime}_{k}=([-\varepsilon b_{1k}]_{+},\ldots, [-\varepsilon b_{(k-1)k}]_{+},-1,[-\varepsilon b_{(k+1)k}]_{+}, \ldots, [-\varepsilon b_{mk}]_{+})^T;$$
  \item[(2)] $F_{\varepsilon}=(f_1,f_2,\ldots,f_{k-1},f^{\prime}_{k},f_{k+1},\ldots,f_n)^{T}$ is the square matrix of degree $n$ with
  $$f_j=(0,\ldots,0,1,0,\ldots,0)\in\ZZ^{n}$$
      for $j\neq k$, and
      $$f^{\prime}_{k}=([\varepsilon b_{k1}]_{+},\ldots, [\varepsilon b_{k(k-1)}]_{+},-1,[\varepsilon b_{k(k+1)}]_{+}, \ldots, [\varepsilon b_{kn}]_{+}).$$
\end{enumerate}
Here $()^{T}$ denotes the transpose of the matrix.

For a compatible pair $(\Lambda,\widetilde{B})$, we denote
$$\Lambda^{\prime}:=E^{T}_{\varepsilon}\Lambda E_{\varepsilon}.$$
It is easy to see that $\Lambda^{\prime}$ is a skew-symmetric matrix. By \cite[Proposition 3.4]{berzel}, the new pair $(\Lambda^{\prime}, B^{\prime})$ is also compatible and $\Lambda^{\prime}$ is independent of the choice of the sign $\varepsilon$. We write $(\Lambda^{\prime},\widetilde{B}^{\prime})= \mu_k(\Lambda,\widetilde{B})$ and say that $(\Lambda^{\prime},\widetilde{B}^{\prime})$ is the mutation of $(\Lambda,\widetilde{B})$ in direction $k$. By \cite[Proposition 3.6]{berzel}, it follows that
$$\mu_{k}(\mu_{k}(\Lambda,\widetilde{B}))=(\Lambda,\widetilde{B}),$$
i.e., $\mu_k$ is an involution.

The skew-symmetric matrix $\Lambda=(\lambda_{ij})$ gives the skew-symmetric bilinear form on the lattice $\ZZ^{m}$ through the mapping
$$\Lambda:\ZZ^{m}\times \ZZ^{m}\longrightarrow \ZZ$$  which sends $(c,d)$ to  $c^T\Lambda d$ for any $c,d\in\ZZ^{m}$.

Let $q$ be a formal variable and  let $\mathbb{Z}[q^{\pm\frac{1}{2}}]\subset \mathbb{Q}(q^{\frac{1}{2}})$ denote the ring of integer Laurent polynomial in the variable $q^{\frac{1}{2}}$.
\begin{definition}
The quantum torus $\mathcal{T}=\mathcal{T}(\Lambda)$ is the $\ZZ[q^{\pm\frac{1}{2}}]$-algebra with a distinguished $\ZZ[q^{\pm\frac{1}{2}}]$-basis $\{X(c)|
c\in \ZZ^{m}\}$ and the multiplication is given by
$$
X(c)X(d)=q^{\frac{1}{2}\Lambda(c,d)}X(c+d)
$$
for any $c,d\in\ZZ^{m}$.
\end{definition}
By following the above definition, we have that
$$
X(c)X(d)=q^{\Lambda(c,d)}X(d)X(c),
$$
$$
X(0)=1~\text{and}~X(-c)=X(c)^{-1}.$$

For each $1\leq i\leq m$, if we set $X(e_i)=X_i$, then
$$
X(c)=q^{{\frac{1}{2}{\tiny{\displaystyle\sum_{l<k}c_kc_l\lambda_{kl}}}}}X_{1}^{c_1}X_{2}^{c_2}\ldots X_{m}^{c_m}
$$
for each $c=(c_1,\ldots,c_m)\in\ZZ^{m}$. The mapping $X: \ZZ^{m} \longrightarrow \mathcal{F} \setminus\{0 \}$ sending any $c$ to $X(c)$ is called a toric frame.

Denote by $\textbf{h}=(\textbf{h}_1,\ldots,\textbf{h}_n)$  with the $i$-th string $\textbf{h}_i=\{h_{i,1}(q^{\frac{1}{2}}),\ldots,h_{i,d_i}(q^{\frac{1}{2}})\}$ for $1\leq i\leq n$, where $h_{i,0}(q^{\frac{1}{2}})=h_{i,d_i}(q^{\frac{1}{2}})=1$ and $h_{i,r}(q^{\frac{1}{2}})$ are  Laurent polynomials in $\ZZ[q^{\pm\frac{1}{2}}]$ satisfying $h_{i,r}(q^{\frac{1}{2}})=h_{i,d_i-r}(q^{\frac{1}{2}})$.

The following definition can be considered as a quantum analogue of Definition \ref{generalized seed}.
\begin{definition}
Let $\widetilde{B}$, $\textbf{h}$, $\Lambda$ and $X$ be described as above. The quadruple $(X,\textbf{h}, \Lambda, \widetilde{B})$ is called a quantum seed if
$(\Lambda,\widetilde{B})$ is a compatible pair.
\end{definition}

Now we are ready to give a quantum analogue of Definition \ref{mutation}.
\begin{definition}
Let $(X, \textbf{h}, \Lambda, \widetilde{B})$ be a quantum seed. For any $1\leq k\leq n$, the new quadruple $\mu_{k}(X, \textbf{h}, \Lambda, \widetilde{B}):=(X^{\prime}, \textbf{h}^{\prime}, \Lambda^{\prime}, \widetilde{B}^{\prime})$ obtained from $(X, \textbf{h}, \Lambda, \widetilde{B})$ in direction $k$ is defined by
\begin{equation}\label{}
    X^{\prime}(e_i)=\begin{cases}
                   X(e_i) & \text{ if $i\neq k$, }\\  \\
                   \sum\limits_{r=0}^{d_k}h_{k,r}(q^{\frac{1}{2}})X(\sum\limits_{j=1}^{m} (r[\beta_{jk}]_{+}+ (d_k-r)[-\beta_{jk}]_{+})e_j -e_k) & \text{ otherwise,}
                  \end{cases}
  \end{equation}
 and
 \begin{equation}
  \textbf{h}^{\prime}=\mu_{k}(\textbf{h}),\ \ \Lambda^{\prime}=\mu_{k}(\Lambda),\ \  \widetilde{B}^{\prime}=\mu_{k}(\widetilde{B}),
  \end{equation}
where $\beta_{ij}=\displaystyle\frac{b_{ij}}{d_j}\in\ZZ$.
\end{definition}

\begin{proposition}\label{prop of (M,B)}
The quadruple $(X^{\prime}, \textbf{h}^{\prime}, \Lambda^{\prime}, \widetilde{B}^{\prime})$ is a quantum seed.
\end{proposition}
\begin{proof}
We need to compute the following relations for any $1\leq i,j\leq m$:
$$X^{\prime}(e_i)X^{\prime}(e_j)=q^{\Lambda^{\prime}(e_i,e_j)} X^{\prime}(e_j)X^{\prime}(e_i).$$
Note that $\Lambda^{\prime}=\mu_k(\Lambda)$. For $i\neq k$ and $j\neq k$,  it follows that $\lambda^{\prime}_{ij}=\lambda_{ij}$. Thus
$$
X^{\prime}(e_i)X^{\prime}(e_j)=q^{\Lambda^{\prime}(e_i,e_j)} X^{\prime}(e_j)X^{\prime}(e_i).
$$

Now we only need to check the following  relation for $i\neq k$:
\begin{equation}\label{commutative rel}
X^{\prime}(e_i)X^{\prime}(e_k)=q^{\Lambda^{\prime}(e_i,e_k)} X^{\prime}(e_k)X^{\prime}(e_i).
\end{equation}

Note that $\Lambda^{\prime}=E_{\varepsilon}^{T}\Lambda E_{\varepsilon}$. Therefore, we have that
\begin{align*}
&X^{\prime}(e_i)X^{\prime}(e_k)=X(e_i)X^{\prime}(e_k)\\
=&\sum\limits_{r=0}^{d_k}h_{k,r}(q^{\frac{1}{2}})X(e_i)X(\sum\limits_{j=1}^{m}(r[\beta_{jk}]_{+}+ (d_k-r)[-\beta_{jk}]_{+})e_j -e_k)\\
=&\sum\limits_{r=0}^{d_k}h_{k,r}(q^{\frac{1}{2}}) q^{\Lambda(e_i,\sum\limits_{j=1}^{m}(r[\beta_{jk}]_{+}+ (d_k-r)[-\beta_{jk}]_{+})e_j -e_k)} \\
& \cdot X(\sum\limits_{j=1}^{m}(r[\beta_{jk}]_{+}+ (d_k-r)[-\beta_{jk}]_{+})e_j -e_k)X(e_i).
\end{align*}

Note that
\begin{align*}
&\Lambda(e_i,\sum\limits_{j=1}^{m}(r[\beta_{jk}]_{+}+ (d_k-r)[-\beta_{jk}]_{+})e_j -e_k)\\
=&\sum\limits_{j=1}^{m}(r\beta_{jk}+d_k[-\beta_{jk}]_{+})\lambda_{ij}-\lambda_{ik} \\ =&\sum\limits_{j=1}^{m}r\beta_{jk}\lambda_{ij}+\sum\limits_{j=1}^{m}[-b_{jk}]_{+}\lambda_{ij}-\lambda_{ik}\\
=&\frac{r}{d_k}\sum\limits_{j=1}^{m}b_{jk}\lambda_{ij}+\sum\limits_{j=1}^{m}[-b_{jk}]_{+}\lambda_{ij}-\lambda_{ik}
\end{align*}

Let $(E_{\varepsilon})_i$ denote the $i$-th column of $E_{\varepsilon}$. Note that
$$
\Lambda^{\prime}(e_i,e_k)=e_{i}^{T}\Lambda(E_{\varepsilon})_k=\Lambda(e_i,\sum\limits_{j=1}^{m}[b_{jk}]_+e_{j}-e_k) =\sum\limits_{j=1}^{m}[b_{jk}]_+ \lambda_{ij}-\lambda_{ik}.
$$
By following the fact  $\sum\limits_{j=1}^{m}b_{jk}\lambda_{ij}=0$, we have that $$\sum\limits_{j=1}^{m}[-b_{jk}]_+\lambda_{ij}= \sum\limits_{j=1}^{m}[b_{jk}]_+\lambda_{ij}.$$

Note that
$$
\frac{r}{d_k}\sum\limits_{j=1}^{m}\lambda_{ij}b_{jk}+\sum\limits_{j=1}^{m}[-b_{jk}]_{+}\lambda_{ij}-\lambda_{ik}= \sum\limits_{j=1}^{m}[-b_{jk}]_{+}\lambda_{ij}-\lambda_{ik}.
$$
Thus $$\Lambda(e_i,\sum\limits_{j=1}^{m}(r[\beta_{jk}]_{+}+ (d_k-r)[-\beta_{jk}]_{+})e_j -e_k)=\Lambda^{\prime}(e_i,e_k).$$
Hence $$
X^{\prime}(e_i)X^{\prime}(e_k)=q^{\lambda^{\prime}_{i,k}} X^{\prime}(e_k)X^{\prime}(e_i).
$$
This completes the proof.
\end{proof}

\begin{proposition}\label{involution}
For each $1\leq k\leq n$, the mutation $\mu_k$ is an involution, i.e.,
$$
\mu_k(\mu_k(X, \textbf{h}, \Lambda, \widetilde{B}))=(X, \textbf{h}, \Lambda, \widetilde{B}).
$$
\end{proposition}
\begin{proof}
It suffices to show that $\mu_k(\mu_k(X(e_i)))=X(e_i)$ for any  $1\leq i \leq n$.

For $i\neq k$, it is easy to see that $\mu_k(\mu_k(X(e_i)))=X(e_i)$.

For $i=k$, we have
\begin{align*}
& X^{\prime}(e_k)X(e_k)\\
=&\sum\limits_{r=0}^{d_k} h_{k,r}(q^{\frac{1}{2}})X(\sum\limits_{j=1}^{m}(r[\beta_{jk}]_{+}+ (d_k-r)[-\beta_{jk}]_{+})e_j-e_k+e_k)\\
&\cdot q^{\frac{1}{2}\Lambda(\sum\limits_{j=1}^{m}(r[\beta_{jk}]_{+}+ (d_k-r)[-\beta_{jk}]_{+})e_j-e_k,e_k)}\\
=&\sum\limits_{r=0}^{d_k} h_{k,r}(q^{\frac{1}{2}})X(\sum\limits_{j=1}^{m}(r\beta_{jk}+ [-b_{jk}]_{+})e_j)
\cdot q^{\frac{1}{2}\Lambda(\sum\limits_{j=1}^{m}(r\beta_{jk}+ [-b_{jk}]_{+})e_j,e_k)}
\end{align*}
and
\begin{align*}
&X^{\prime}(e_k)X^{\prime\prime}(e_k)\\
=&\sum\limits_{r=0}^{d_k} h_{k,r}(q^{\frac{1}{2}})X^{\prime} (e_k+\sum\limits_{j=1}^{m}(r[\beta_{jk}^{\prime}]_{+}+ (d_k-r)[-\beta_{jk}^{\prime}]_{+})e_j-e_k)\\
&\cdot q^{\frac{1}{2}\Lambda^{\prime}(e_k,\sum\limits_{j=1}^{m}(r[\beta^{\prime}_{jk}]_{+}+ (d_k-r)[-\beta^{\prime}_{jk}]_{+})e_j-e_k)}\\
=&\sum\limits_{r=0}^{d_k} h_{k,r}(q^{\frac{1}{2}})X^{\prime} (\sum\limits_{j=1}^{m}(r[\beta_{jk}^{\prime}]_{+}+ (d_k-r)[-\beta_{jk}^{\prime}]_{+})e_j)\\
&\cdot q^{\frac{1}{2}\Lambda^{\prime}(e_k,\sum\limits_{j=1}^{m}(r[\beta^{\prime}_{jk}]_{+}+ (d_k-r)[-\beta^{\prime}_{jk}]_{+})e_j)}.
\end{align*}

By using  $h_{k,r}(q^{\frac{1}{2}})=h_{k,d_k-r}(q^{\frac{1}{2}})$,  we have that

\begin{align*}
&X^{\prime}(e_k)X^{\prime\prime}(e_k)\\
=&\sum\limits_{r=0}^{d_k} h_{k,d_k-r}(q^{\frac{1}{2}})X^{\prime} (\sum\limits_{j=1}^{m}((d_k-r)[\beta_{jk}^{\prime}]_{+}+ r[-\beta_{jk}^{\prime}]_{+})e_j)\\
&\cdot q^{\frac{1}{2}\Lambda^{\prime}(e_k,\sum\limits_{j=1}^{m}((d_k-r)[\beta^{\prime}_{jk}]_{+}+ r[-\beta^{\prime}_{jk}]_{+})e_j)}\\
=&\sum\limits_{r=0}^{d_k} h_{k,r}(q^{\frac{1}{2}})X^{\prime}(\sum\limits_{j=1}^{m}(-r\beta_{jk}^{\prime}+ [b_{jk}^{\prime}]_{+})e_j)
\cdot q^{\frac{1}{2}\Lambda^{\prime}(e_k,\sum\limits_{j=1}^{m}(-r\beta^{\prime}_{jk}+ [b^{\prime}_{jk}]_{+})e_j)}.
\end{align*}

For $j\neq k$,  we have that $$b_{jk}^{\prime}=\sum\limits_{i=1}^{n}(b_{ji}f_{ik}+[-\varepsilon b_{jk}]_{+}b_{ki}f_{ik})=-b_{jk}-[-\varepsilon b_{jk}]_{+}b_{kk}=-b_{jk}.$$ Thus $b_{jk}^{\prime}=-b_{jk}$ and $\beta_{jk}^{\prime}=-\beta_{jk}$. It follows that
$$
\sum\limits_{j=1}^{m}(r\beta_{jk}+ [-b_{jk}]_{+})e_j= \sum\limits_{j=1}^{m}(-r\beta_{jk}^{\prime}+ [b_{jk}^{\prime}]_{+})e_j.
$$
Recall that $X(a_1,\ldots,a_m)=q^{\frac{1}{2}\sum_{i<j}a_ia_j\lambda_{ji}}X_{1}^{a_1}\ldots X_{m}^{a_m}$. We get
$$
X(\sum\limits_{j=1}^{m}(r\beta_{jk}+ [-b_{jk}]_{+})e_j)= X^{\prime}(\sum\limits_{j=1}^{m}(-r\beta_{jk}^{\prime}+ [b_{jk}^{\prime}]_{+})e_j).
$$

Since $\Lambda^{\prime}=E_{+}^{T}\Lambda E_+$, we get
$$
\lambda^{\prime}_{kj}=(\sum\limits_{i=1}^{m}[-b_{ik}]_+ e_{i}^{T}-e_{k}^{T})\Lambda e_j=\Lambda(\sum\limits_{j=1}^{m}[-b_{ik}]_+ e_{i}-e_k,e_j).
$$
Thus
\begin{align*}
&\Lambda^{\prime}(e_k,\sum\limits_{j=1}^{m}(-r\beta_{jk}^{\prime}+ [b_{jk}^{\prime}]_{+})e_j)\\
=&\sum\limits_{j=1}^{m}(r\beta_{jk}+ [-b_{jk}]_{+})\lambda^{\prime}_{kj}\\
=&\sum\limits_{j=1}^{m}(r\beta_{jk}+ [-b_{jk}]_{+})\Lambda(\sum\limits_{i=1}^{m}[-b_{ik}]_{+}e_i-e_k,e_j)\\
=&\Lambda(\sum\limits_{i=1}^{m}[-b_{ik}]_{+}e_i-e_k,\sum\limits_{j=1}^{m}(r\beta_{jk}+ [-b_{jk}]_{+})e_j)\\
=&\Lambda(\sum\limits_{i=1}^{m}[-b_{ik}]_{+}e_i,\sum\limits_{j=1}^{m}(r\beta_{jk}+ [-b_{jk}]_{+})e_j) +\Lambda(\sum\limits_{j=1}^{m}(r\beta_{jk}+ [-b_{jk}]_{+})e_j,e_k).
\end{align*}
Note that
\begin{align*}
&\Lambda(\sum\limits_{i=1}^{m}[-b_{ik}]_{+}e_i,\sum\limits_{j=1}^{m}(r\beta_{jk}+ [-b_{jk}]_{+})e_j) \\
=&\Lambda(\sum\limits_{i=1}^{m}[-b_{ik}]_{+}e_i,\sum\limits_{j=1}^{m}r\beta_{jk} e_j)\\
=&\frac{r}{d_k}\Lambda(\sum\limits_{i=1}^{m}[-b_{ik}]_{+}e_i,\sum\limits_{j=1}^{m}b_{jk}e_j) \\
=&\frac{r}{d_k}(\sum\limits_{j=1}^{m}[-b_{jk}]_{+}e_{j}^{T})\Lambda(\widetilde{B})_{k}\\
=&\frac{r}{d_k}(\sum\limits_{j=1}^{m}[-b_{jk}]_{+}e_{j}^{T})(-\widetilde{d}_ke_k)\\=&0,
\end{align*}
where $(\widetilde{B})_{k}$ denotes the $k$-th column of $\widetilde{B}$. Therefore
$$
\Lambda^{\prime}(e_k,\sum\limits_{j=1}^{m}(-r\beta_{jk}^{\prime}+ [b_{jk}^{\prime}]_{+})e_j) =\Lambda(\sum\limits_{j=1}^{m}(r\beta_{jk}+ [-b_{jk}]_{+})e_j,e_k).
$$
Hence $X^{\prime}(e_k)X(e_k)=X^{\prime}(e_k)X^{\prime\prime}(e_k)$ from which we deduce that $X(e_k)=X^{\prime\prime}(e_k)$.
\end{proof}

According to Proposition \ref{involution}, the following relation on quantum seeds is an equivalence relation:
$(X^{\prime}, \textbf{h}^{\prime}, \Lambda^{\prime}, \widetilde{B}^{\prime})$ is called to be mutation-equivalent to $(X, \textbf{h}, \Lambda, \widetilde{B})$,
if $(X^{\prime}, \textbf{h}^{\prime}, \Lambda^{\prime}, \widetilde{B}^{\prime})$ can be obtained from
$(X, \textbf{h}, \Lambda, \widetilde{B})$ by a sequence of seed mutations, i.e.
$$(X^{\prime}, \textbf{h}^{\prime}, \Lambda^{\prime}, \widetilde{B}^{\prime})=\mu_{k_t}(\ldots(\mu_{k_1}(X, \textbf{h}, \Lambda, \widetilde{B}))\ldots)$$ for some $1\leq k_1,\ldots,k_t\leq n$. The set
$$
\{X^{\prime}(e_1),\ldots,X^{\prime}(e_n)\}
$$
is called a cluster of $(X^{\prime}, \textbf{h}^{\prime}, \Lambda^{\prime}, \widetilde{B}^{\prime})$ and $X^{\prime}(e_i)$ are called cluster variables for $1\leq i\leq n$.

Now we can define generalized quantum cluster algebras as follows.

\begin{definition}
The generalized quantum cluster algebra $\mathcal{A}(X, \textbf{h}, \Lambda, \widetilde{B})$ associated with the initial seed $(X, \textbf{h}, \Lambda, \widetilde{B})$, is the $\ZZ[q^{\pm\frac{1}{2}}][X_{n+1}^{\pm1},\dots,X_{m}^{\pm1}]$-subalgebra of $\Fcal$ generated by the cluster variables from the seeds which are mutation-equivalent to $(X, \textbf{h}, \Lambda, \widetilde{B})$.
\end{definition}

\begin{remark} We have that
\begin{enumerate}
      \item[(1)] if $d_k=1$ for all $1\leq k\leq n$, then the generalized quantum cluster algebra $\mathcal{A}(X,\textbf{h},\Lambda,\widetilde{B})$ is exactly the quantum cluster algebra introduced by Berenstein and Zelevinsky \cite{berzel};
 \item[(2)] if $q=1$, then the generalized quantum cluster algebra $\mathcal{A}(X, \textbf{h}, \Lambda, \widetilde{B})$ is exactly a class of  generalized cluster algebras introduced in the beginning of this section.
 \end{enumerate}

\end{remark}

We conclude this section by considering the following simplest nontrivial example of a generalized quantum
cluster algebra.
\begin{example}(Type $\emph{B}_2$)
Let $\mathcal{A}(1,2)$ denote the generalized quantum cluster algebra associated with the compatible pair $(\Lambda,\widetilde{B})$, where $\textbf{d}=(2,1)$,
\begin{equation*}
\Lambda=\left(
  \begin{array}{cc}
    0 & 1 \\
    -1 & 0 \\
  \end{array}
\right)
\text{~and~}
B=\widetilde{B}=\left(
  \begin{array}{cc}
    0 & 1 \\
    -2 & 0 \\
  \end{array}
\right).
\end{equation*}

The generalized quantum cluster algebra $\mathcal{A}(1,2)$ is the $\ZZ[q^{\pm\frac{1}{2}}]$-subalgebra of $\Fcal$ generated by $\{X_i~|~i\in\ZZ\}$, where the cluster variables $X_i$ are given by the following exchange relations
\begin{equation}\label{}
   X_{k-1}X_{k+1}=\begin{cases}
                   1+q^{\frac{1}{2}}X_k & \text{ if $k$ is odd,}\\
                   1+q^{\frac{1}{2}}h(q^{\frac{1}{2}})X_k+qX_{k}^{2} & \text{ if $k$ is even,}
                  \end{cases}
\end{equation}
for any $h(q^{\frac{1}{2}})\in\ZZ[q^{\pm\frac{1}{2}}]$.

For for $a_1, a_2 \in \mathbb{Z}$, denote by $X(a_1,a_2):=q^{-\frac{a_1a_2}{2}}X_{1}^{a_1}X_{2}^{a_2}$.  We can compute all cluster variables  as follows:
\begin{enumerate}
\item[] $
X_3=X(-1,-2)+X(-1,0)+h(q^{\frac{1}{2}})X(-1,1);
$
\item[]$
X_4=X(0,-1)+X(-1,1)+h(q^{\frac{1}{2}})X(-1,0)+X(-1,-1);
$
\item[]
$X_5=X(1,-2)+(q^{-\frac{1}{2}}+q^{\frac{1}{2}})X(0,-2)+X(-1,-2)+X(-1,0)\\
{\hspace{1.0cm}}+h(q^{\frac{1}{2}})X(-1,-1)+h(q^{\frac{1}{2}})X(0,-1);$
\item[]$
X_6=X(0,-1)+X(1,-1);
$
\item[]$
X_7=X_1;
$
\item[]$
X_8=X_2.
$
\end{enumerate}

\end{example}

It follows  that  the Laurent phenomenon is true for type $\emph{B}_2$.

\section{Generalized quantum cluster algebras of rank two}

In this section we will prove that the Laurent phenomenon hold true for generalized quantum cluster algebras of rank two.

Consider the following based quantum torus (see \cite{berzel} for more details)
$$\mathcal{T}=\mathbb{Z}[q^{\pm\frac{1}{2}}][X_{1}^{\pm 1},X_{2}^{\pm 1}|X_{1}X_{2}=qX_{2}X_{1}].$$
 In this section, we denote by $\mathcal{F}$ the skew-field of fractions of the based quantum torus $\mathcal{T}$.

Let $P_1(x),P_2(x)\in \mathbb{Z}[q^{\pm\frac{1}{2}}][x]$ be  the  polynomials of arbitrary positive degree $d_1$ and $d_2$, respectively. Both $P_1(x)$ and $P_2(x)$ have  the form
 $$P(x)=1+q^{\frac{1}{2}}h_{1}(q^{\frac{1}{2}})x+qh_{2}(q^{\frac{1}{2}})x^{2}+\cdots+q^{\frac{d-1}{2}}h_{d-1}(q^{\frac{1}{2}})x^{d-1}
 +q^{\frac{d}{2}}x^{d}$$
 where $h_{i}(x)\in \mathbb{Z}[x^{\pm 1}]$
 satisfies that  $h_{i}(q^{\frac{1}{2}})=h_{d-i}(q^{\frac{1}{2}})$  for any $1\leq i\leq d$. Sometimes, we also make use of the notations $h_{0}(x)=h_{d}(x)=1$ without causing any confusion.

We inductively define $X_{k}\in\mathcal{F}$ for $k\in\mathbb{Z}$ from the following exchange relations
\begin{equation}\label{eq:exchange relation}
   X_{k-1}X_{k+1}=\begin{cases}
                   P_1(X_k) & \text{ if $k$ is even;}\\
                   P_2(X_k) & \text{ if $k$ is odd.}
                  \end{cases}
  \end{equation}

\begin{definition}
The  generalized quantum cluster algebra $\mathcal{A}_{q}(P_1,P_2)$ is defined to be the $\mathbb{Z}[q^{\pm\frac{1}{2}}]$-subalgebra of $\mathcal{F}$ generated by the set of all cluster variables $\{X_k\}_{k\in\mathbb{Z}}$.
\end{definition}
\begin{remark}
When $q=1$, $\mathcal{A}_{q}(P_1,P_2)$ is degenerated to the  generalized cluster algebra of rank two studied in \cite{rupel}.
\end{remark}

One can easily check the following result by induction.
\begin{lemma}\label{lemma1} In $\mathcal{A}_{q}(P_1,P_2)$, for any $k\in\mathbb{Z}$ we have that
$$X_{k}X_{k+1}=qX_{k+1}X_{k}.$$
\end{lemma}
The next result will be  useful for us to prove the quantum Laurent phenomenon.
\begin{lemma}\label{lemma2}
In $\mathcal{A}_{q}(P_1,P_2)$, we have that

\begin{enumerate}
      \item[(1)] If $k\in\mathbb{Z}$ is even, then we have
$$X_{k+1}X_{k-1}=1+q^{-\frac{1}{2}}h_{1}(q^{\frac{1}{2}})X_k+q^{-1}h_{2}(q^{\frac{1}{2}})X_k^{2}+
\cdots+q^{-\frac{d_1}{2}}h_{d_1}(q^{\frac{1}{2}})X_k^{d_1};$$
\item[(2)] if $k\in\mathbb{Z}$ is odd, then we have
$$X_{k+1}X_{k-1}=1+q^{-\frac{1}{2}}h'_{1}(q^{\frac{1}{2}})X_k+q^{-1}h'_{2}(q^{\frac{1}{2}})X_k^{2}+
\cdots+q^{-\frac{d_2}{2}}h'_{d_2}(q^{\frac{1}{2}})X_k^{d_2}.$$
\end{enumerate}
\end{lemma}
\begin{proof}
We only prove $(2)$, the proof of $(1)$ is similar. Assume that $k$ is odd. According to the exchange relation and Lemma \ref{lemma1}, we have that
\begin{align*}
   & X_{k+1}=X_{k-1}^{-1}P_2(X_k)\\
           =&X_{k-1}^{-1}+q^{\frac{1}{2}}h'_{1}(q^{\frac{1}{2}})X_{k-1}^{-1}X_{k}+q^{1}h'_{2}(q^{\frac{1}{2}})X_{k-1}^{-1}X_{k}^{2}+
\cdots+q^{\frac{d_2}{2}}h'_{d_2}(q^{\frac{1}{2}})X_{k-1}^{-1}X_{k}^{d_2}\\
           =&X_{k-1}^{-1}+q^{-\frac{1}{2}}h'_{1}(q^{\frac{1}{2}})X_{k}X_{k-1}^{-1}+q^{-1}h'_{2}(q^{\frac{1}{2}})X_{k}^{2}X_{k-1}^{-1}+
\cdots+q^{-\frac{d_2}{2}}h'_{d_2}(q^{\frac{1}{2}})X_{k}^{d_2}X_{k-1}^{-1}.\\
   \end{align*}
   It follows that
   $$X_{k+1}X_{k-1}=1+q^{-\frac{1}{2}}h'_{1}(q^{\frac{1}{2}})X_k+q^{-1}h'_{2}(q^{\frac{1}{2}})X_k^{2}+
\cdots+q^{-\frac{d_2}{2}}h'_{d_2}(q^{\frac{1}{2}})X_k^{d_2},$$
which proves the desired result.
\end{proof}
 It will be convenient to introduce the following notations:
\begin{equation}\label{eq:exchange relation}
   X_{k+1}X_{k-1}=\begin{cases}
                   \widehat{P_1(X_k)} & \text{ if $k$ is even;}\\
                   \widehat{P_2(X_k)} & \text{ if $k$ is odd.}
                  \end{cases}
  \end{equation}
Note that $\widehat{P_1(X_k)}$ and $\widehat{P_2(X_k)}$ have the explicit expressions respectively obtained in Lemma \ref{lemma2}.

Now we are ready to prove the quantum Laurent phenomenon of  rank two quantum generalized cluster algebra  by using the method developed in \cite{rupel}.
\begin{theorem}(Quantum Laurent phenomenon)\label{laurent}
The generalized quantum cluster algebra $\mathcal{A}_{q}(P_1,P_2)$ is a
subalgebra of $\mathbb{Z}[q^{\pm\frac{1}{2}}][X_{m}^{\pm 1},X_{m+1}^{\pm 1}]$ for any $m\in\mathbb{Z}$.
\end{theorem}
\begin{proof}
Firstly for any $k\in\mathbb{Z},$ we prove that
$$X_{k}\in \mathbb{Z}[q^{\pm\frac{1}{2}}][X_{k+1},X_{k+2},X_{k+3},X_{k+4}].$$
We only prove the statement for  the case when $k$ is odd, the even case can be proved similarly.
Assume that $k$ is odd. Now by using Lemma \ref{lemma2}, we compute the element $X_{k+1}^{d_1}X_{k+4}$ as follows:
   \begin{align*}
   & X_{k+1}^{d_1}X_{k+4} \\
    =&X_{k+1}^{d_1}\widehat{P_1(X_{k+3})}X_{k+2}^{-1}\\
    =&(X_{k+1}^{d_1}\widehat{P_1(X_{k+3})}-q^{-\frac{d_1}{2}}P_1(X_{k+1}))X_{k+2}^{-1}
    +q^{-\frac{d_1}{2}}P_1(X_{k+1})X_{k+2}^{-1}\\
    =&(\sum_{i=0}^{d_1}X_{k+1}^{d_1}q^{-\frac{i}{2}}h_{i}(q^{\frac{1}{2}})X_{k+3}^{i}-q^{-\frac{d_1}{2}}\sum_{i=0}^{d_1}
    q^{\frac{d_1-i}{2}}h_{d_1-i}(q^{\frac{1}{2}})X_{k+1}^{d_1-i})X_{k+2}^{-1}+q^{-\frac{d_1}{2}}X_{k}\\
    =&[\sum_{i=0}^{d_1}q^{-\frac{i}{2}}h_{i}(q^{\frac{1}{2}})X_{k+1}^{d_1-i}(X_{k+1}^{i}X_{k+3}^{i}-1)]X_{k+2}^{-1}+q^{-\frac{d_1}{2}}X_{k}.
   \end{align*}
Note that $X_{k+1}X_{k+3}=P_2(X_{k+2})$ which has $1$ as the constant term, so by Lemma \ref{lemma1}, we can deduce that
$$X_{k+1}^{i}X_{k+3}^{i}-1 \in X_{k+2}\mathbb{Z}[q^{\pm\frac{1}{2}}][X_{k+2}], \text{ for any } 0\leq i\leq d_1.$$ Thus we have that
$$X_{k}\in \mathbb{Z}[q^{\pm\frac{1}{2}}][X_{k+1},X_{k+2},X_{k+3},X_{k+4}].$$

By induction, we can deduce that $X_{k}\in \mathbb{Z}[q^{\pm\frac{1}{2}}][X_{m+1},X_{m+2},X_{m+3},X_{m+4}]$ for any $m\geq k.$ Similarly, we have
$X_{k}\in \mathbb{Z}[q^{\pm\frac{1}{2}}][X_{m-4},X_{m-3},X_{m-2},X_{m-1}]$ for any $k\geq m.$  It follows that
$X_{k}\in \mathbb{Z}[q^{\pm\frac{1}{2}}][X_{m-1},X_{m},X_{m+1},X_{m+2}]$ for any $m\in\mathbb{Z}.$ Note that $X_{m-1}$ and $X_{m+2}$ belong to
$\mathbb{Z}[q^{\pm\frac{1}{2}}][X_{m}^{\pm 1},X_{m+1}^{\pm 1}]$, hence the proof is finished.
\end{proof}
Let $X\rightarrow \overline{X}$ be the $\mathbb{Z}-$linear bar-involution of the based quantum torus $\mathcal{T}$ (see \cite{berzel}) satisfying
$$\overline{q^{\frac{r}{2}}X(a_1,a_2)}=q^{-\frac{r}{2}}X(a_1,a_2), \ \ \ r,a_1,a_2\in \mathbb{Z},$$
where the notation $X(a_1,a_2):=q^{-\frac{a_1a_2}{2}}X_{1}^{a_1}X_{2}^{a_2}$.

In order to prove the  bar-invariance of quantum generalized cluster variables, we need more conditions on coefficients in $P(x)$ as stated in the next result.
\begin{proposition}
If the coefficients in $P(x)$ satisfies  $\overline{h_{i}(q^{\frac{1}{2}})}=h_{i}(q^{\frac{1}{2}})$ for each $1\leq i\leq d$,  then all quantum generalized cluster variables of $\mathcal{A}_{q}(P_1,P_2)$ are invariant under the bar-involution.
\end{proposition}
\begin{proof}
By Lemma \ref{lemma2} and the condition $h_{i}(q^{\frac{1}{2}})=h_{i}(q^{-\frac{1}{2}})$ for each $1\leq i\leq d$, we have
\begin{equation}\label{eq:exchange relation2}
   X_{k+1}X_{k-1}=\begin{cases}
                   \overline{P_1(X_k)} & \text{ if $k$ is even;}\\
                   \overline{P_2(X_k)} & \text{ if $k$ is odd.}
                  \end{cases}
  \end{equation}
  Using the bar-involution to the exchange relation, and noting the quantum Laurent phenomenon established in Theorem \ref{laurent}, we can deduce
  \begin{equation}\label{eq:exchange relation2}
   \overline{X_{k+1}}\ \overline{X_{k-1}}=\begin{cases}
                   \overline{P_1(X_k)} & \text{ if $k$ is even;}\\
                   \overline{P_2(X_k)} & \text{ if $k$ is odd.}
                  \end{cases}
  \end{equation}
  Hence the proof can be completed by induction on $k$.
\end{proof}
As a direct corollary, we obtain the following result.
\begin{corollary}
The quantum generalized cluster algebra $\mathcal{A}_{q}(P_1,P_2)$ is invariant under the bar-involution.
\end{corollary}
According to the proof of Theorem \ref{laurent}, we have that
$$\mathcal{A}_{q}(P_1,P_2)=\mathbb{Z}[q^{\pm\frac{1}{2}}][X_{m-1},X_{m},X_{m+1},X_{m+2}]$$ for any $m\in\mathbb{Z}.$
This is an analogue  of the fact that the classical cluster algebra of rank two is equal to its lower bound. The following definition is an analogue of Definition 1.15 in \cite{bfz}.
\begin{definition}
A standard monomial in the quantum generalized cluster variables $\{X_0,X_1,X_2,X_3\}$ is an element of the form
$X_1^{a_1}X_2^{a_2}X_3^{a'_1}X_0^{a'_2}$, where all exponents are nonnegative integers with $a_1a'_1=0$ and $a_2a'_2=0$.
\end{definition}
By the exchange relations among the elements $\{X_0,X_1,X_2,X_3\}$ and the equation $\mathcal{A}_{q}(P_1,P_2)=\mathbb{Z}[q^{\pm\frac{1}{2}}][X_{0},X_{1},X_{2},X_{3}]$, we can easily deduce the following result.
\begin{proposition}
The set of all standard monomials in the quantum generalized cluster variables $\{X_0,X_1,X_2,X_3\}$ is a $\mathbb{Z}[q^{\pm\frac{1}{2}}]-$basis
of $\mathcal{A}_{q}(P_1,P_2).$
\end{proposition}
\begin{remark}
The $\mathbb{Z}[q^{\pm\frac{1}{2}}]-$basis consisting of all standard monomials is not invariant under the bar-involution. How to construct various
bar-invariant positive $\mathbb{Z}[q^{\pm\frac{1}{2}}]-$bases of $\mathcal{A}_{q}(P_1,P_2)$ does deserve a further study.
\end{remark}

We denote by $\mathcal{U}(P_1,P_2)\subset \mathcal{F}$ the $\mathbb{Z}[q^{\pm\frac{1}{2}}]-$subalgebra of $\mathcal{F}$
given by
$$\mathcal{U}(P_1,P_2)=\mathbb{Z}[q^{\pm\frac{1}{2}}][X_{1}^{\pm 1},X_{2}^{\pm 1}]\cap \mathbb{Z}[q^{\pm\frac{1}{2}}][X_{2}^{\pm 1},X_{3}^{\pm 1}]
\cap \mathbb{Z}[q^{\pm\frac{1}{2}}][X_{0}^{\pm 1},X_{1}^{\pm 1}].$$

Then we have a stronger version of the quantum Laurent phenomenon.
\begin{theorem} For generalized quantum cluster algebras of rank two, we have that
$$\mathcal{A}_{q}(P_1,P_2)=\bigcap_{k\in\mathbb{Z}}\mathbb{Z}[q^{\pm\frac{1}{2}}][X_{k}^{\pm 1},X_{k+1}^{\pm 1}]=\mathcal{U}(P_1,P_2).$$
\end{theorem}
\begin{proof}
By Theorem \ref{laurent}, we have that $$\mathcal{A}_{q}(P_1,P_2)\subseteq\bigcap_{k\in\mathbb{Z}}\mathbb{Z}[q^{\pm\frac{1}{2}}][X_{k}^{\pm 1},X_{k+1}^{\pm 1}]\subseteq\mathcal{U}(P_1,P_2).$$ So we only need to prove that $\mathcal{U}(P_1,P_2)\subseteq \mathcal{A}_{q}(P_1,P_2)$.
In fact, one can prove a  stronger result as follows
 $$\bigcap_{k=m-1}^{m+1}\mathbb{Z}[q^{\pm\frac{1}{2}}][X_{k}^{\pm 1},X_{k+1}^{\pm 1}]= \mathcal{A}_{q}(P_1,P_2).$$
This can be finished by showing  the following four equations for any $m\in\mathbb{Z}$ which are similar to those in \cite{berzel}:
\begin{enumerate}
      \item[(1)]  $\mathbb{Z}[q^{\pm\frac{1}{2}}][X_{m}^{\pm 1},X_{m+1}^{\pm 1}]\cap \mathbb{Z}[q^{\pm\frac{1}{2}}][X_{m+1}^{\pm 1},X_{m+2}^{\pm 1}]
=\mathbb{Z}[q^{\pm\frac{1}{2}}][X_{m},X_{m+1}^{\pm 1},X_{m+2}]$;
\item[(2)] $\mathbb{Z}[q^{\pm\frac{1}{2}}][X_{m},X_{m+1}^{\pm 1}]\cap \mathbb{Z}[q^{\pm\frac{1}{2}}][X_{m-1},X_{m}^{\pm 1},X_{m+1}]
=\mathbb{Z}[q^{\pm\frac{1}{2}}][X_{m-1},X_{m},X_{m+1}]$;
 \item[(3)] $\mathbb{Z}[q^{\pm\frac{1}{2}}][X_{m},X_{m+1}^{\pm 1},X_{m+2}]=\mathbb{Z}[q^{\pm\frac{1}{2}}][X_{m-1},X_{m},X_{m+1},X_{m+2}]
+\mathbb{Z}[q^{\pm\frac{1}{2}}][X_{m},X_{m+1}^{\pm 1}]$;
 \item[(4)] $\mathbb{Z}[q^{\pm\frac{1}{2}}][X_{m},X_{m+1}^{\pm 1},X_{m+2}]\cap \mathbb{Z}[q^{\pm\frac{1}{2}}][X_{m-1},X_{m}^{\pm 1},X_{m+1}]
=\mathbb{Z}[q^{\pm\frac{1}{2}}][X_{m-1},X_{m},X_{m+1},X_{m+2}]$.
\end{enumerate} Here we omit the details of the proofs of these four equations.

It follows that  \begin{align*}
   & \bigcap_{k=m-1}^{m+1}\mathbb{Z}[q^{\pm\frac{1}{2}}][X_{k}^{\pm 1},X_{k+1}^{\pm 1}]\\
    =&\mathbb{Z}[q^{\pm\frac{1}{2}}][X_{m-1}^{\pm 1},X_{m}^{\pm 1}]\cap \mathbb{Z}[q^{\pm\frac{1}{2}}][X_{m}^{\pm 1},X_{m+1}^{\pm 1}]
            \cap \mathbb{Z}[q^{\pm\frac{1}{2}}][X_{m}^{\pm 1},X_{m+1}^{\pm 1}]\cap \mathbb{Z}[q^{\pm\frac{1}{2}}][X_{m+1}^{\pm 1},X_{m+2}^{\pm 1}]\\
           =&\mathbb{Z}[q^{\pm\frac{1}{2}}][X_{m-1},X_{m}^{\pm 1},X_{m+1}]\cap \mathbb{Z}[q^{\pm\frac{1}{2}}][X_{m},X_{m+1}^{\pm 1},X_{m+2}]\\
           =&\mathbb{Z}[q^{\pm\frac{1}{2}}][X_{m-1},X_{m},X_{m+1},X_{m+2}]\\
           =&\mathcal{A}_{q}(P_1,P_2).
   \end{align*}
This completes the proof of the theorem.
\end{proof}



\begin{thebibliography}{99}

\bibitem{bfz}
A.~Berenstein, S.~Fomin, and A.~Zelevinsky, \emph{Cluster algebras III: Upper bounds and double Bruhat cells,} Duke Math. J. 126 (2005), no. 1, 1--52.


\bibitem{berzel}
A.~Berenstein and A.~Zelevinsky, \emph{Quantum cluster algebras,}
Adv. Math. \textbf{195} (2005), 405--455.



\bibitem{CS}
L.~Chekhov and M.~Shapiro, \emph{Teichm¡§uller spaces of Riemann surfaces with orbifold
points of arbitrary order and cluster variables,} Int. Math. Res. Notices 2014 (2014),
2746--2772.



\bibitem{ca1}
S.~Fomin and A.~Zelevinsky, \emph{Cluster algebras. I. Foundations,}
J. Amer. Math. Soc.  \textbf{15}  (2002),  no. 2, 497--529.

\bibitem{ca2}
S.~Fomin and A.~Zelevinsky, \emph{Cluster algebras. II. Finite type
classification},  Invent. Math.  \textbf{154}  (2003),  no. 1,
63--121.

\bibitem{fg1}
V. V. Fock and A. B. Goncharov, \emph{Cluster ensembles, quantization and the dilogarithm,}
Annales Sci. \'{E}c. Norm. Sup\'{e}r. 42 (2009), 865--930.

\bibitem{fg2}
V. V. Fock and A. B. Goncharov, \emph{The quantum dilogarithm and representations of quantum cluster varieties,}
Invent. Math. 172 (2009), 223--286.

\bibitem{gsv16}
M. Gekhtman, M. Shapiro and A. Vainshtein, \emph{Drinfeld double of $GL_n$ and generalized cluster structures,}
2016, arXiv:1605.05705 [math.QA].

\bibitem{gleitz14}
A. Gleitz, \emph{Quantum affine algebras at roots of unity and generalised cluster algebras,}
2014, arXiv:1410.2446 [math.RT].

\bibitem{nak0}
T.~Nakanishi, \emph{Quantum generalized cluster algebras and quantum dilogarithms
of higher degrees,} Theor. Math. Phys. 185 (2015), 1759--1768.

\bibitem{nak}
T.~Nakanishi, \emph{Structure of seeds in generalized cluster algebras,}  Pacific J. Math. 277 (2015),
201--218.

\bibitem{rupel}
D.~Rupel, \emph{Greedy bases in rank 2 generalized cluster algebras,} 2013, arXiv:1309.2567 [math.RA].

\bibitem{usn}
A.~Usnich, \emph{Non-commutative Laurent phenomenon for two variables,} 2010, arXiv:1006.1211 [math.AG].


\end{thebibliography}
\end{document}